\documentclass[a4paper, 12pt]{amsart}

\usepackage{amsmath,amsthm,amssymb,enumerate,graphicx,overpic}

\usepackage[usenames,dvipsnames,svgnames]{xcolor}
\usepackage[margin=1.2in]{geometry}
\usepackage{hyperref} 
\hypersetup{colorlinks=true,
allcolors=ForestGreen} 

\DeclareMathOperator{\mcg}{MCG}
\DeclareMathOperator{\pmcg}{PMCG}
\DeclareMathOperator{\Torelli}{\mathcal{I}}
\DeclareMathOperator{\Ends}{Ends}
\newcommand{\ccmcg}[1]{\overline{\mcg_c(#1)}}
\newcommand{\C}{\mathcal{C}}
\newcommand{\R}{\mathbb{R}}

\newcommand{\N}{\mathbb{N}}
\newcommand{\Y}{\mathcal{Y}}
\newcommand{\st}{\;|\;}
\newcommand{\ssm}{\smallsetminus}

\newtheorem{thm}{Theorem}
\newtheorem{lemma}[thm]{Lemma}
\newtheorem{prop}[thm]{Proposition}

\newtheorem{thmintro}{Theorem}

\newtheorem{qst}[thmintro]{Question}

\theoremstyle{definition}

\title{Multitwists in big mapping class groups}
\author{George Domat}
\email{george.domat@rice.edu}
\address{Rice University, Department of Mathematics, 6100 Main St.--MS 136, Houston, TX 77005, USA}
\author{Federica Fanoni}
\email{federica.fanoni@u-pec.fr}
\address{CNRS, Univ Paris Est Creteil, Univ Gustave Eiffel, LAMAR8050, F-94010 Creteil, France}
\author{Sebastian Hensel}
\email{hensel@math.lmu.de}
\address{Mathematisches Institut der Universit\"at M\"unchen, Theresienstr. 39, 80333 M\"unchen, Germany}
\date{\today}

\begin{document}
\begin{abstract}
We show that the group generated by multitwists (i.e.\ products of powers of twists about disjoint non-accumulating curves) doesn't contain the Torelli group of an infinite-type surface. As a consequence, multitwists don't generate the closure of the compactly supported mapping class group of a surface of infinite type.
\end{abstract} 
\maketitle
\section{Introduction}
The mapping class group of a surface of finite type has been thoroughly studied for decades. In particular, multiple \emph{simple} sets of generators are known. The Dehn--Lickorish theorem (\cite{dehn_gruppe}, \cite{lickorish_finite}), in combination with the Birman exact sequence (\cite{birman_mapping}), shows that the pure mapping class group of a finite-type surface can be generated by finitely many Dehn twists about nonseparating curves, and we need to add finitely many half-twists to generate the full mapping class group. Humphries \cite{humphries_generators} proved that, if the surface is closed and of genus \(g\geq 2\), \(2g+1\) Dehn twists about nonseparating curves suffice to generate the mapping class group, and moreover this number is optimal: fewer than \(2g+1\) Dehn twists cannot generate. Other results show that mapping class groups can be generated by two elements (see e.g.\ \cite{wajnryb_mapping}), by finitely many involutions or by finitely many torsion elements (see e.g.\ \cite{bf_every}).

In the case of surfaces of infinite type, the (pure) mapping class group is uncountable, so in particular it is not finitely (nor countably) generated. For some of these surfaces the mapping class group is generated by torsion elements, or even by involutions (see \cite{mt_self-similar} and \cite{cc_normal}), while for other surfaces they aren't (see \cite{domat_big}, \cite{mt_self-similar}, \cite{cc_normal}). To the best of our knowledge, no other generating set is known.

Note that the (pure) mapping class group of a surface of infinite type is endowed with an interesting topology, induced by the compact-open topology on the group of homeomorphisms of the surface. So it is interesting to talk about \emph{topological} generating sets (sets whose \emph{closure} of the group they generate is the (pure) mapping class group). It follows from the finite-type results that Dehn twists topologically generate the closure of the compactly supported mapping class group. Moreover, Patel and Vlamis \cite{pv_algebraic} proved that the pure mapping class group of a surface is topologically generated by Dehn twists if the surface has at most one nonplanar end, and by Dehn twists and maps called \emph{handle shifts} otherwise.

The goal of this note is to investigate a natural candidate for a set of generators of the closure of the compactly supported mapping class group of a surface: the collection of \emph{multitwists}. A multitwist is a (possibly infinite) product of powers of Dehn twists about a collection of simple closed curves that do not accumulate anywhere in the surface. Our main result is a negative one, and it follows from a non-generation result for the Torelli group:

\begin{thmintro}\label{thm:main}
Let \(S\) be an infinite-type surface. Then the subgroup of the mapping class group of $S$ generated by multitwists doesn't contain the Torelli group. In particular, multitwists don't generate the closure of the compactly supported mapping class group.
\end{thmintro}

The idea of the proof is to produce an explicit element in the Torelli group that is not in the subgroup generated by multitwists. This element is built by taking an infinite product of increasing powers of partial pseudo-Anosov homeomorphisms supported on disjoint finite-type subsurfaces. We use work of Bestvina, Bromberg and Fujiwara \cite{bbf_stable} to certify that the mapping class we construct is not in the subgroup generated by multitwists.

Theorem \ref{thm:main} also begs the following question: 

\begin{qst}
	What is the subgroup generated by the collection of multitwists? Is there an alternative, more explicit description of its elements?
\end{qst}

Furthermore, our theorem shows that the subgroup generated by the collection of multitwists is not a closed subgroup of the mapping class group. Therefore, it does not immediately inherit a Polish topology from the topology on the mapping class group.

\begin{qst}
	Is the subgroup generated by the collection of multitwists a Polish group?
\end{qst}

\subsection*{Acknowledgements}
The authors would like to thank Mladen Bestvina for his suggestion of how to remove an unnecessary assumption in the main theorem. They are also grateful to the organizers of the \emph{Big Mapping Class Groups and Diffeomorphism Groups} conference, during which most of the work was done. They thank the referees for their useful comments. The second author thanks Peter Feller for useful conversations.

The first author was supported in part by the Fields Institute for Research in Mathematical Sciences and NSF RTG--1745670.

\section{Preliminaries}
In this note, a surface is a connected, orientable, Hausdorff, second countable two-dimensional manifold, without boundary unless otherwise stated. One notable exception is any subsurface, which will always have compact boundary. Boundary components of subsurfaces are assumed to be homotopically nontrivial, but are allowed to be homotopic to a puncture.

Surfaces are \emph{of finite type} if their fundamental groups are finitely generated and \emph{of infinite type} otherwise. A surface \(S\) is \emph{exceptional} if it has genus zero and at most four punctures or genus one and at most one puncture, otherwise it is \emph{nonexceptional}.

The \emph{mapping class group} of a surface \(S\) is the group \(\mcg(S)\) of orientation preserving homeomorphisms of \(S\) up to homotopy. The \emph{pure mapping class group} \(\pmcg(S)\) is the subgroup of \(\mcg(S)\) fixing all ends and ---if there are any --- boundary components, and \(\ccmcg{S}\) denotes the closure of the subgroup generated by compactly supported mapping classes. The \emph{Torelli group} $\Torelli(S)$ is the subgroup of the mapping class group given by elements acting trivially on the first homology group of the surface.

A pseudo-Anosov mapping class \(f\) of a finite-type surface is \emph{chiral} if \(f^k\) is not conjugate to \(f^{-k}\) for every \(k\neq 0\). Note that it follows from \cite{bbf_stable} that every pseudo-Anosov element in the Torelli group is chiral: indeed, by \cite[Theorem 5.6]{bbf_stable}, every nontrivial element in the Torelli group has positive stable commutator length, and by \cite[Theorem 4.2]{bbf_stable} any pseudo-Anosov element with positive stable commutator length is chiral.

A \emph{curve} on a surface is the homotopy class of an essential (i.e.\ not homotopic to a point, a puncture or a boundary component) simple closed curve.
Given a curve \(\alpha\), we denote by \(\tau_\alpha\) the Dehn twist about \(\alpha\).

An \emph{integral weighted multicurve} \(\mu\) is a formal sum \(\sum_{i\in I}n_i\alpha_i\), where the \(\alpha_i\) are pairwise disjoint curves not accumulating anywhere and the \(n_i\) are integers. Given an integral weighted multicurve \(\mu\), we define \(\tau_\mu\) to be the mapping class
\[\tau_\mu=\prod_{i\in I}\tau_{\alpha_i}^{n_i}.\]
Such a mapping class is called a \emph{multitwist}.

We say that an integral weighted multicurve is \emph{finite} if \(I\) is finite (i.e.\ it contains finitely many curves). An integral weighted multicurve \(\nu\) is a \emph{submulticurve} of an integral weighted multicurve \(\mu=\sum_{i\in I}n_i\alpha_i\) if \(\nu=\sum_{i\in J}n_i\alpha_i\), where \(J\subset I\).

Given a surface with boundary, an \emph{arc} is the homotopy class (relative to the boundary) of a simple arc that cannot be homotoped into the boundary. We denote by \(\C(S)\) the \emph{curve and arc graph} of a surface \(S\), where vertices are curves and, if \(\partial S\neq \emptyset\), arcs, and two vertices are adjacent if they have disjoint representatives.

For any two subsurfaces $A$ and $B$ of $S$ that have an essential intersection, the \emph{subsurface projection} of $B$ to $A$ is the subset $\partial B \cap A \subset \C(A)$. This projection is denoted by $\pi_{A}(B)$. For any $\beta \in \C(B)$ we also define $\pi_{A}(\beta) := \pi_{A}(B)$. These projections always have bounded diameter \cite{mm_geometryII} and given any intersecting subsurfaces $A,B,C \subset S$ we define the \emph{projection distance} as:
\begin{align*}
	d_{A}(B,C) := \operatorname{diam}_{\C(A)}(\pi_{A}(B) \cup \pi_{A}(C)).
\end{align*}

For a subgroup $G < \mcg(S)$, we say that a subsurface, $K \subset S$, is \emph{$G$-nondisplaceable} if $gK$ and $K$ cannot be homotoped to be disjoint for any $g \in G$. Note that if a subsurface $K$ is $G$-nondisplaceable, subsurface projections are always defined between $G$-translates of $K$.

\section{Proof of Theorem \ref{thm:main}}

Fix an infinite-type surface \(S\), different from the Loch Ness monster.

Figure \ref{fig:nondsubsurfaces} shows some examples of $\ccmcg{S}$-nondisplaceable subsurfaces that will be used in the following lemma. 

\begin{lemma}\label{lem:nondisp}
If \(S\) is an infinite-type surface, different from the Loch Ness monster, it contains infinitely many finite-type \(\ccmcg{S}\)-nondisplaceable subsurfaces, which are pairwise disjoint, pairwise homeomorphic and non-accumulating. Moreover, the subsurfaces can be chosen to be nonexceptional.
\end{lemma}

\begin{proof}
The proof is a case-by-case analysis. In each case we describe a finite-type \(\ccmcg{S}\)-nondisplaceable subsurface such that we can clearly find infinitely many copies with the required properties.

\underline{Case 1:} \(S\) is the once-punctured Loch Ness monster. Then note that any separating curve \(\alpha\) which separates the two ends cannot be mapped disjointly from itself by any mapping class (because it bounds a nondisplaceable subsurface). As a consequence, for any \(g\geq 1\), any genus-\(g\) subsurface with two boundary components separating the two ends is nondisplaceable.

\underline{Case 2:} \(S\) has at least two nonplanar ends. Note that by the argument in \cite[Proposition 6.3]{pv_algebraic}, any separating curve such that both complementary components have infinite genus is \(\ccmcg{S}\)-nondisplaceable.
\begin{itemize}
\item If \(S\) has at least one nonplanar end --- denoted \(e\) --- which is isolated in \(\Ends(S)\), for any \(g\geq 1\), any genus-\(g\) subsurface with two separating boundary components, each of which cuts off a surface containing only the end \(e\), is \(\ccmcg{S}\)-nondisplaceable.
\item If \(S\) has at least one nonplanar end --- denoted \(e\) --- which is isolated in \(\Ends_g(S)\) but not in \(\Ends(S)\), for any \(g\geq 1\), any genus-\(g\) subsurface with three separating boundary components, two of which cut off a subsurface whose only nonplanar end is \(e\) and the third one cuts off a planar surface, is \(\ccmcg{S}\)-nondisplaceable.
\item If no nonplanar end is isolated in \(\Ends_g(S)\), \(\Ends_g(S)\) is a Cantor set. If it contains an end \(e\) that is not accumulated by planar ends, we choose a genus-\(g\) subsurface with two separating boundary components and no planar ends, so that each complementary component has infinite genus, for \(g\geq 1\). Otherwise, we choose a genus-\(g\) subsurface with three separating boundary components, so that two complementary components have infinite genus and one is a planar subsurface, for \(g\geq 1\).
\end{itemize}

\underline{Case 3:} \(S\) has no nonplanar ends. We can then choose any \(n\)-holed sphere whose boundary curves are separating, so that there is at least one end in each complementary component, for \(n\geq 5\).
\end{proof}

\begin{figure}[h]
\begin{center}
\includegraphics[width=.8\textwidth]{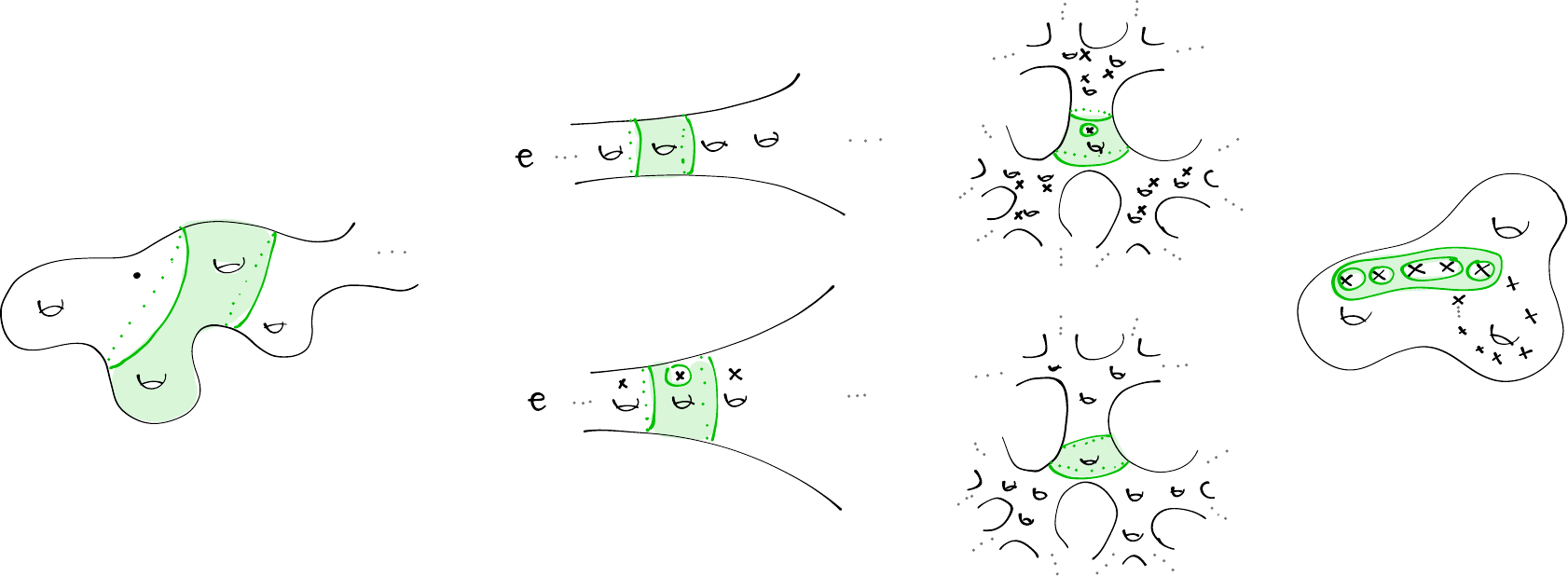}
\caption{The subsurfaces of Lemma \ref{lem:nondisp}}
\label{fig:nondsubsurfaces}
\end{center}
\end{figure}

Fix a finite-type \(\ccmcg{S}\)--nondisplaceable subsurface \(\Sigma\subset S\) and let \(\Y\) be the \(\ccmcg{S}\)-orbit of \(\Sigma\). As \(\Sigma\) is \(\ccmcg{S}\)-nondisplaceable, any two surfaces in \(\Y\) have intersecting boundaries --- in particular, subsurface projections \(\pi_A\) between surfaces in \(\Y\) are always defined. Moreover, by \cite{behrstock_asymptotic} and \cite{mm_geometryII}, there is some constant \(\mu>0\) so that for every \(A,B,C\in\Y\):
\begin{itemize}
\item at most one of \(d_A(B,C)\), \(d_B(A,C)\) and \(d_C(A,B)\) is bigger than \(\mu\), and
\item \(|\{D\in\Y\st d_D(A,B)>\mu\}|\) is finite.
\end{itemize}
See \cite[Lemma 3.8]{domat_big} for details on checking these in the infinite-type case.
We can therefore run the projection complex machinery to deduce (see \cite[Proposition 2.7]{bbf_stable}):

\begin{prop}
\(\ccmcg{S}\) acts by isometries on a hyperbolic graph \(\C(\Y)\) so that for every \(A,B\in\Y\), \(A\neq B\):
\begin{enumerate}
\item \(\C(A)\) is isometrically embedded as a convex set in \(\C(\Y)\) and the images of \(\C(A)\) and \(\C(B)\) are disjoint;
\item the inclusion
\[\bigsqcup_{C\in\Y}\C(Y)\hookrightarrow \C(\Y)\]
is \(\ccmcg{S}\)-equivariant;
\item the nearest point projection to \(\C(A)\) sends \(\C(B)\) to a bounded set, which is at uniformly bounded distance from \(\pi_A(B)\);
\item if \(g\in \ccmcg{S}\) is supported on \(A\) and the restriction is pseudo-Anosov, and \(\Gamma\) is the subgroup of \(\ccmcg{S}\) given by elements leaving \(A\) invariant and preserving the stable and unstable foliations of \(g\), then \((\ccmcg{S},\C(\Y),g,\Gamma)\) satisfies WWPD.
\end{enumerate}
\end{prop}

Furthermore, the same proof as \cite[Lemma 2.8]{bbf_stable} yields:

\begin{lemma}\label{lem:finitemtwists}
Let \(\tau\) be a multitwist about a finite multicurve \(\mu\). Then for every \(A\in\Y\) there is a vertex \(v_A\) of \(\C(\Y)\) such that the nearest point projection to \(\C(A)\) of the \(\tau\) orbit of \(v_A\) is uniformly bounded. In particular, if \(\tau\) is hyperbolic, its virtual quasi-axis can intersect \(\C(A)\) only in a bounded length segment.
\end{lemma}

\begin{proof}
If \(\mu\cap A=\emptyset\), \(\tau\) fixes any element of \(\C(A)\), so it's elliptic. Otherwise, let \(v_A\) be an element of \(A\cap\mu\neq \emptyset\).
Then the nearest point projection of $\tau^{n}(v_{A})$ to $\C(A)$ is a uniformly bounded distance from \(\pi_A(\tau^n(v_A))\), which is defined to be \(\pi_A(\partial(\tau^n(A))\).  But this is at bounded distance from \(\tau^n(\mu)\cap A=\mu\cap A\), so the projection of \(\tau^n(v_A)\) is at uniformly bounded distance from \(A\cap \mu\) for every \(n\). This proves the first statement of the lemma. The second statement follows as in the proof of \cite[Lemma 2.8]{bbf_stable}.
\end{proof}

As a consequence, we can apply \cite[Proposition 3.1]{bbf_stable} to deduce:

\begin{prop}\label{prop:quasimorphism}
Let \(\Sigma\) be a \(\ccmcg{S}\)-nondisplaceable subsurface of finite type and  \(f\) a mapping class that is a pure chiral pseudo-Anosov mapping class of \(\Sigma\) of sufficiently large translation length and the identity on the complement. Then there is a homogeneous quasimorphism \(\varphi:\ccmcg{S}\to\R\) of defect \(\Delta\) such that \(|\varphi(f^n)|\to \infty\) and \(|\varphi(\tau)|\leq \Delta\) for every multitwist \(\tau\).
\end{prop}

\begin{proof}
The only thing we need to check is that \(|\varphi(\tau)|\leq \Delta\) for every multitwist \(\tau\). But a multitwist \(\tau\) associated to an integral weighted multicurve \(\mu=\sum_{i\in I}n_i\alpha_i\) can be written as a product of two multitwists, \(\tau_1\) and \(\tau_2\), where \(\tau_1\) is associated to the integral weighted multicurve
\[\mu_1=\sum_{i:\alpha_i\cap\Sigma\neq \emptyset}n_i\alpha_i\]
and \(\tau_2\) to the integral weighted multicurve
\[\mu_2=\sum_{i:\alpha_i\cap\Sigma= \emptyset}n_i\alpha_i.\]
Then \(\tau_2\) acts elliptically on \(\C(\Y)\), for \(\Y=\ccmcg{S}\cdot\Sigma\), so \(\varphi(\tau_2)=0\). By Lemma \ref{lem:finitemtwists}, if \(\tau_1\) doesn't act elliptically on \(\C(\Y)\), its virtual quasi-axis has small projections. Thus, by how the quasimorphism $\phi$ is constructed in \cite[Proposition 3.1]{bbf_stable}, if the translation length of $f$ is larger than the projection bound from Lemma \ref{lem:finitemtwists}, we have \(\varphi(\tau_1)=0\). As a consequence \(|\varphi(\tau)|\leq \Delta\). 
\end{proof}

\begin{proof}[Proof of Theorem \ref{thm:main}]	
	 Suppose first that \(S\) is not the Loch Ness monster. We will  construct an element $F$ of $\Torelli(S)$ which is not a finite product of multitwists. Since $F$ is in $\overline{\mcg_c(S)}$ (by construction, or by the fact that $\Torelli(S)\subset \overline{\mcg_c(S)}$ by \cite{agkmtw_big}), this will also show that multitwists don't generate the closure of the compactly supported mapping class group. By Lemma \ref{lem:nondisp}, we can find pairwise disjoint and nonaccumulating finite-type subsurfaces \(\Sigma_n\), all pairwise homeomorphic and \(\ccmcg{S}\)-nondisplaceable. Fix \(\Sigma\) a surface homeomorphic to the \(\Sigma_n\) and \(\theta_n:\Sigma\to\Sigma_n\) a homeomorphism. Choose a pure chiral pseudo-Anosov mapping class \(f\) in the Torelli group of \(\Sigma\), and let \(F_n\) be the mapping class of \(S\) with support on \(\Sigma_n\) and so that \(F_n|_{\Sigma_n}=\theta_n\circ f\circ\theta_n^{-1}\). Then for any \(n\), by Proposition \ref{prop:quasimorphism} (after potentially passing to a power in order to increase the translation length), we can find a homogeneous quasimorphism \(\varphi_n:\ccmcg{S}\to \R\) with defect \(\Delta\) (independent of \(n\)) that is unbounded on powers of \(F_n\) and bounded by $\Delta$ on all multitwists or elements acting elliptically on \(\C(\ccmcg{S}\cdot\Sigma_n)\). Choose powers \(k_n\) so that
\[|\varphi_n(F_n^{k_n})|\to\infty,\]
which exist because by assumption \(|\varphi_n(F_n)|>1\) for every \(n\). Define
\[F=\prod_{n\in\N}F_n^{k_n}.\]
For every \(n\), \(\prod_{m\neq n}F_m^{k_m}\) acts elliptically on \(\C(\ccmcg{S}\cdot\Sigma_n)\), so 
\[|\varphi_n(F)|=\left|\varphi_n\left(F_n^{k_n}\circ \prod_{m\neq n}F_m^{k_m}\right)\right|\geq \left|\varphi_n\left(F_n^{k_n}\right)\right|-\Delta\to\infty.\]
	 
If \(F\) were a product of \(k\) multitwists \(\tau_1,\dots,\tau_k\), then for any \(n\), \[|\varphi_n(F)|=|\varphi_n(\tau_k\circ\dots\circ\tau_1)|\leq 2k\Delta,\] which gives a contradiction.

Suppose now that \(S\) is the Loch Ness monster and fix a point \(x\in S\). By the Birman exact sequence \cite[Appendix]{domat_big}, the kernel of the surjection \[\Torelli(S\ssm \{x\})\to\Torelli(S)\] is the fundamental group of $S$ and is therefore generated by twists. In particular, if \(\Torelli(S)\) is contained in the subgroup generated by multitwists, so is the Torelli group of the once-punctured Loch Ness monster, a contradiction. By \cite{pv_algebraic}, $\mcg(S)=\overline{\mcg_c(S)}$ and $\mcg(S\ssm\{x\})=\overline{\mcg_c(S\ssm\{x\})}$, so the same argument applied to $\mcg(S)$ and $\mcg(S\ssm\{x\})$ proves the result for the closure of the compactly supported mapping class group.

\end{proof}

\bibliographystyle{alpha}
\bibliography{references}

\newcommand{\etalchar}[1]{$^{#1}$}
\begin{thebibliography}{AGK{\etalchar{+}}19}

\bibitem[AGK{\etalchar{+}}19]{agkmtw_big}
Javier Aramayona, Tyrone Ghaswala, Autumn~E. Kent, Alan McLeay, Jing Tao, and
  Rebecca~R. Winarski.
\newblock Big {T}orelli groups: generation and commensuration.
\newblock {\em Groups Geom. Dyn.}, 13(4):1373--1399, 2019.

\bibitem[BBF16]{bbf_stable}
Mladen Bestvina, Ken Bromberg, and Koji Fujiwara.
\newblock Stable commutator length on mapping class groups.
\newblock {\em Ann. Inst. Fourier (Grenoble)}, 66(3):871--898, 2016.

\bibitem[Beh06]{behrstock_asymptotic}
Jason~A. Behrstock.
\newblock Asymptotic geometry of the mapping class group and {T}eichm\"{u}ller
  space.
\newblock {\em Geom. Topol.}, 10:1523--1578, 2006.

\bibitem[BF04]{bf_every}
Tara~E. Brendle and Benson Farb.
\newblock Every mapping class group is generated by 6 involutions.
\newblock {\em J. Algebra}, 278(1):187--198, 2004.

\bibitem[Bir69]{birman_mapping}
Joan~S. Birman.
\newblock Mapping class groups and their relationship to braid groups.
\newblock {\em Comm. Pure Appl. Math.}, 22:213--238, 1969.

\bibitem[CC22]{cc_normal}
Danny Calegari and Lvzhou Chen.
\newblock Normal subgroups of big mapping class groups.
\newblock {\em Trans. Amer. Math. Soc. Ser. B}, 9:957--976, 2022.

\bibitem[Deh38]{dehn_gruppe}
M.~Dehn.
\newblock Die {G}ruppe der {A}bbildungsklassen.
\newblock {\em Acta Math.}, 69(1):135--206, 1938.
\newblock Das arithmetische Feld auf Fl\"{a}chen.

\bibitem[Dom22]{domat_big}
George Domat.
\newblock Big pure mapping class groups are never perfect.
\newblock {\em MRL}, 29(3):691--726, 2022.
\newblock Appendix with Ryan Dickmann.

\bibitem[Hum79]{humphries_generators}
Stephen~P. Humphries.
\newblock Generators for the mapping class group.
\newblock In {\em Topology of low-dimensional manifolds ({P}roc. {S}econd
  {S}ussex {C}onf., {C}helwood {G}ate, 1977)}, volume 722 of {\em Lecture Notes
  in Math.}, pages 44--47. Springer, Berlin, 1979.

\bibitem[Lic64]{lickorish_finite}
W.~B.~R. Lickorish.
\newblock A finite set of generators for the homeotopy group of a
  {$2$}-manifold.
\newblock {\em Proc. Cambridge Philos. Soc.}, 60:769--778, 1964.

\bibitem[MM00]{mm_geometryII}
H.~A. Masur and Y.~N. Minsky.
\newblock Geometry of the complex of curves. {II}. {H}ierarchical structure.
\newblock {\em Geom. Funct. Anal.}, 10(4):902--974, 2000.

\bibitem[MT21]{mt_self-similar}
Justin {Malestein} and Jing {Tao}.
\newblock {Self-similar surfaces: involutions and perfection}.
\newblock {\em arXiv e-prints}, page arXiv:2106.03681, June 2021.

\bibitem[PV18]{pv_algebraic}
Priyam Patel and Nicholas~G. Vlamis.
\newblock Algebraic and topological properties of big mapping class groups.
\newblock {\em Algebr. Geom. Topol.}, 18(7):4109--4142, 2018.

\bibitem[Waj96]{wajnryb_mapping}
Bronislaw Wajnryb.
\newblock Mapping class group of a surface is generated by two elements.
\newblock {\em Topology}, 35(2):377--383, 1996.

\end{thebibliography}

\end{document}